\documentclass[12pt,reqno]{amsart}
\usepackage{setspace}
\usepackage{amsmath, amscd, graphicx, latexsym, times, }
\usepackage{mathtools}  

\usepackage{tikz-cd}  

\usepackage{xparse}
\usepackage{float}
\usepackage{hyperref}
\usepackage{mathtools}

\newtheorem{theorem}{Theorem}[section]
\newtheorem*{theorem-non}{Theorem}
\newtheorem{prop}[theorem]{Proposition}

\newtheorem{remark}[theorem]{Remark}

\newtheorem{cor}[theorem]{Corollary}

\newcommand\C{{\mathbb C}}

\newcommand\CPP{{\mathbb{CP}}^1}

\def \M {\mathcal M}
\def \om {\overline {\M}}
\def \Hb {H_{\bullet}}
\def \Cb {H^{\bullet}}

\def \bd {\partial}

\def \Q {\mathbb{Q}}

\def \DD {\Delta}

\def\fg{\mathfrak{g}}

\title[On the BV structure on the cohomology of moduli space]{On the
  BV structure on the cohomology of moduli space}

\dedicatory{Dedicated to Kyoji Saito on the occasion of his 75th
  birthday}

\begin{document}

\author{S\"{u}meyra Sakall{\i}}
\address{Max Planck Institute for Mathematics,
Vivatsgasse 7, 53111 Bonn,
Germany}
\email{sakal008@umn.edu}

\author{Alexander A. Voronov}
\address {School of Mathematics\\University of Minnesota\\
  Minneapolis, MN 55455, USA, and Kavli IPMU (WPI), UTIAS, University
  of Tokyo, Kashiwa, Chiba 277-8583, Japan}
\email{voronov@umn.edu}

\date{December 29, 2019}

\begin{abstract}
The question of vanishing of the BV operator on the cohomology of the
moduli space of Riemann surfaces is investigated. The BV structure,
which comprises a BV operator and an antibracket, is identified,
vanishing theorems are proven, and a counterexample is provided.
\end{abstract}


\maketitle

\section*{Introduction}
Since the early nineties, there has been considerable interest in
nonperturbative methods in string field theory, in particular, setting
up a \emph{quantum master equation $($QME$)$}
\begin{equation}
  \label{QME}
dS + \hbar \DD S + \frac{1}{2} \{S,S\} = 0
\end{equation}
in the dg-BV-algebra $C_c^{\bullet}(\M)$ of compactly supported
cochains with rational coefficients on the moduli space $\M$ of
Riemann surfaces and describing a solution $S$ of the QME, see, for
example, \cite{Witt,Zw, Sen-Zw, KSV, Eliashberg-Givental-Hofer, Markl,
  CL, Cos, HVZ, Kauf, JM, MuSA, KWZ, DM, CFL}. One work that caught
the eye was that of Kevin Costello \cite{Cos}, in which he used
homotopical-algebraic methods and certain elementary facts about the
cohomology of moduli spaces to prove the existence and uniqueness, up
to homotopy, of a solution $S$ of the QME. The proof consisted in a
clever reduction of the QME to a linear equation
\begin{equation}
  \label{cocycle}
\hat{d} S' = 0
\end{equation}
on a related $S'$ in a homotopy abelian subquotient of
$C_c^{\bullet}(\M)$ for the differential $\hat{d} = d + \hbar \DD$ and
transferring the problem to cohomology.

Note that the QME \eqref{QME} is actually a \emph{Maurer-Cartan
  equation}
\[
\hat{d}S + \frac{1}{2} \{S,S\} = 0
\]
in the dg-Lie algebra of cochains on the moduli space and if this
dg-Lie algebra is homotopy abelian, then the Maurer-Cartan equation is
equivalent to the cocycle equation \eqref{cocycle}. It is known, see
\cite{terilla:smoothness, KKP}, that in a dg-BV-algebra, the homotopy
abelianness of the underlying dg-Lie algebra may be derived from the
degeneration at $E_1$ of a spectral sequence associated to the double
complex $(C_c^{\bullet}(\M), d, \DD)$, which entails the vanishing of
the differential $d_1 = \DD$ on the cohomology $H_c^{\bullet}(\M)$.

Given Costello's homotopy abelianness result, one might hope that
homotopy abelianness takes place for the whole dg-Lie algebra
$C_c^{\bullet}(\M)$ and that the general argument of
\cite{terilla:smoothness, KKP} would imply Costello's existence and
uniqueness result. On the other hand, any general statement, such as
the vanishing of a cohomology operation, $\DD$ in particular, on the
cohomology of the moduli space is extremely interesting, given how
little is still known about it.

This paper grew out of our investigation of the vanishing of $\DD$ on
the cohomology $H_c^{\bullet}(\M)$ of the moduli space. It has turned
out that $\DD$ vanishes on vast ranges of these cohomology groups,
appearing to provide convincing evidence for vanishing everywhere, see
Section \ref{vanishing} below. On the other hand, we have found out
that total vanishing does not actually take place. We provide a
counterexample in Corollary \ref{counter}. In the process, we obtain
useful description of the BV operator $\DD$ and the antibracket $\{-,
-\}$, relating them to the differential $d_1$ in the spectral sequence
associated to the topological filtration on the compactified moduli
space $\om$. This differential may be expressed as the Gysin
homomorphism given by Poincar\'{e} residue, see Section
\ref{identification}. It is somewhat surprising that the differentials
on the first terms of the spectral sequences associated to the double
complex $(C_c^{\bullet}(\M), d, \DD)$ and the topological filtration
on $\om$ are closely related, see Theorem \ref{delta-gysin}.

\subsection*{Conventions}

\begin{sloppypar}
In this paper, $\M_{g,n}$ denotes the (usual) moduli space of smooth,
connected, compact, genus $g$ Riemann surfaces with $n$ labeled
punctures. Here and throughout the paper, we assume that our Riemann
surfaces are \emph{stable}, \emph{i.e}., $n \ge 3$ if $g = 0$ and $n
\ge 1$ if $g = 1$; otherwise $g, n \ge 0$. Let $S_n$ be the group of
permutations of $n$ objects. Then $\M_g (n) := \M_{g,n} / S_n$ is the
moduli space of compact Riemann surfaces with $n$ unlabeled punctures.
\end{sloppypar}
  
We work with homology $H_{\bullet}(X) := H_{\bullet}(X,\Q)$ and
cohomology $H^{\bullet}(X) := H^{\bullet}(X,\Q)$ with rational
coefficients. Vector spaces are also assumed to be over $\Q$.

\subsection*{Acknowledgments}
With admiration and gratitude, we dedicate this paper to Kyoji
Saito. The breadth of his interests touched upon the topics of this
paper as well: his recent paper \cite{LLS} uses the BV-algebra of
multivector fields on a Calabi-Yau manifold, whereas his interest in
the moduli spaces is long-standing, see \cite{saito:moduli,
  saito:higher, saito:the, saito:algebraic}. We are grateful to Weiyan
Chen, Domenico D'Ales\-san\-dro, S\o ren Galatius, and Craig
Westerland for useful discussions. SS also thanks the Max Planck
Institute for Mathematics, Bonn, for its great research environment
and support. AV was supported by the World Premier International
Research Center Initiative (WPI), MEXT, Japan, and a Collaboration
grant from the Simons Foundation (\#585720).

\section{The BV structure on the (co)homology of moduli spaces}

\subsection{Operations on the homology $\Hb (\M_g(n),\Q)$}

The following operations have been originally defined by Zwiebach
\cite{Zw} and modified by Costello \cite{Cos}.

1. For $n \ge 0$, define the \emph{BV operator}
\begin{equation}
\Delta: H_k(\M_g(n+2)) \rightarrow H_{k+1}(\M_{g+1}(n))
\label{delta}
\end{equation}
as follows. If $Z \in H_k(\M_g(n+2))$ is a homology class represented
by a singular cycle $Z = \sum_{i=1}^m c_i \sigma_i$ for $\sigma_i :
\Delta^k \rightarrow \M_g(n+2)$, then set
\[
\Delta(Z) := \sum_i c_i \sum_{\{\beta, \gamma\}} (\text{twist-gluing
  of $\sigma_i$ at punctures $\{\beta, \gamma\}$}\big),
\]
where the summation runs over unordered pairs $\{\beta, \gamma\}$ of
punctures and {\em twist-gluing} is cutting out small holomorphic
disks around the punctures $\{\beta, \gamma\}$, and gluing their
complement(s) along the boundaries. More precisely, we consider
holomorphic disks at each puncture (i.e., holomorphic embeddings of
the standard disk $|z| < 1$ centered at the puncture and not
containing other punctures). Then we cut out the disks $|z| \leq r$
and $|w| \leq r$ for some $r = 1 - \epsilon$ at sewn punctures and
identify the annuli $r < |z| < 1/r$ and $r < |w| < 1/r$ via $w =
e^{it}/z$, with $t$ running over the interval $[0,2\pi]$, and thereby
increasing the degree of the chain $\sigma_i$. This gives a new chain
$S^1 \times \Delta^k \rightarrow \M_{g+1}(n)$, resulting in a new
cycle $\Delta(Z)$. Moreover, the map $\Delta$ is well-defined on
$H_{\bullet}(\M_g(n+2))$, \emph{i.e}., it is independent of the choice
of holomorphic disks involved in twist-gluing (see \cite{Zw, Cos}).



2. Similarly, for $n_1, n_2 \ge 0$, define the \emph{antibracket}
\begin{multline}
\label{bracket}
\{-, -\} : H_{k_1}(\M_{g_1}(n_1+1)) \otimes H_{k_2}
(\M_{g_2}(n_2+1))\\ \rightarrow H_{k_1 + k_2 +1}(\M_{g_1 + g_2}(n_1 +
n_2))
\end{multline}
as follows. If $Z_1 \in H_{k_1}(\M_{g_1}(n_1+1))$ and $Z_2 \in
H_{k_2}(\M_{g_2}(n_2+1))$ are homology classes represented by singular
cycles $Z_1 = \sum_{i=1}^m c_i \sigma_i$ for $\sigma_i : \Delta^{k_1}
\rightarrow \M_{g_1}(n_1+1)$ and $Z_2 = \sum_{j=1}^m d_j \tau_j$ for
$\tau_j : \Delta^{k_2} \rightarrow \M_{g_2}(n_2+1)$, then set
\[
\{Z_1, Z_2\} := \sum_{i,j} c_i d_j \sum_{\beta, \gamma}
\left(\text{twist-gluing of $\sigma_i$ with $\tau_j$ at punctures
  $\beta, \gamma$}\right),
\]
where $\beta$ runs over the punctures of the surface in
$\M_{g_1}(n_1+1)$ and $\gamma$ runs over the punctures of the surface
in $\M_{g_2}(n_2+1)$. This gives a new chain $S^1 \times \Delta^{k_1}
\times \Delta^{k_2} \rightarrow \M_{g_1 + g_2}(n_1 + n_2)$, resulting
in a new cycle $\{Z_1, Z_2\} \in H_{k_1 + k_2+1}(\M_{g_1 + g_2}(n_1 +
n_2))$, independent of the choices made along the way (see ibid.).

The reason why these homology operations are introduced is to set up a
quantum master equation, actually, at the chain, rather than homology
level:
\[
dS + \hbar \Delta S + \frac{1}{2} \{S,S\} = 0
\]
on a formal power series $S = \sum_{\substack{g,n \ge 0\\ 2g - 2 + n >
    0}} S_{g,n} \lambda^{2g-2+n} \hbar^g$ with coefficients $S_{g,n}$
being chains in $C_{6g-6+2n} (\M_g(n))$ and $d$ denoting chain
boundary. However, for the above equation to be sensible, degree
considerations suggest to change grading on chains and homology to
grading by codimension, as well as assume that the formal variables
$\lambda$ and $\hbar$ have degree zero. This change of grading may be
regarded as the application of canonical \emph{Poincar\'e-Lefschetz
  duality}
\[
H_k(\M_g(n)) = H^{6g-6+2n-k}_c (\M_g(n)),
\]
where $H^\bullet_c$ denotes cohomology with compact support. Thus,
under Poincar\'e-Lefschetz duality, the BV operator \eqref{delta} and
antibracket \eqref{bracket} turn into
\begin{equation}
  \label{delta-c}
\Delta: H^k_c (\M_g(n+2)) \rightarrow H^{k+1}_c (\M_{g+1}(n))
\end{equation}
and
\begin{multline}
\label{bracket-c}
\{-, -\} : H^{k_1}_c(\M_{g_1}(n_1+1)) \otimes H^{k_2}_c
(\M_{g_2}(n_2+1))\\ \rightarrow H^{k_1 + k_2 +1}_c (\M_{g_1 + g_2}(n_1
+ n_2)),
\end{multline}
respectively. We hope that, given that Poincar\'e-Lefschetz duality
identifies these operations \eqref{delta-c} and \eqref{bracket-c} with
\eqref{delta} and \eqref{bracket}, respectively, our use of the same
notation, $\Delta$ and $\{-, -\}$, will not create serious confusion.

From \cite{barannikov,Zw} we know that these two operations define the
structure of a dg-Lie algebra with a differential $\Delta$ of degree
$1$ on
\begin{equation}
\label{g}
\fg := \bigoplus_{\substack{g,n \ge 0\\ 2g - 2 + n > 0}}
H^{\bullet}_c(\M_{g}(n))[1],
\end{equation}
where $V[1]$ denotes the desuspension of a graded vector space $V$:
$V[1]^k := V^{k+1}$.

\begin{remark}
\label{physics}
The graded symmetric algebra $S(\fg[-1])$ on the suspension of a
dg-Lie algebra $\fg$ is known to carry the structure of a
\emph{BV-algebra}, see \cite[Example 2.5]{QDT2}, which uses a slightly
different grading convention: $\deg \Delta = -1$ and $\deg \hbar =
2$. In the case of our dg-Lie algebra $\fg$ as in \eqref{g}, the
graded symmetric algebra $S(\fg[-1])$ is isomorphic to the
(co)homology $H^\bullet_c (\M)$ of Zwiebach's dg-BV-algebra
$C^\bullet_c (\M)$ of compactly supported chains in the moduli space
$\M$ of not necessarily connected, closed Riemann surfaces with
unlabeled punctures.
\end{remark}

\subsection{Construction via the real Deligne-Mumford compactification}

Let $\underline{\M}_{g}(n)$ be the real Deligne-Mumford
compactification of $\M_{g}(n)$ to an orbifold with corners (see
\cite{KSV}). It is known that $\underline{\M}_{g}(n)$ is homotopy
equivalent to $\M_{g}(n)$ and $H_{\bullet}(\underline{\M}_{g}(n))
\cong H_{\bullet}(\M_{g}(n))$. In this construction, $\Delta$ uses
twist-attaching (in which there are $S^1$ ways of attaching a pair of
punctures on a Riemann surface $\Sigma \in \underline{\M}_{g}(n)$, or
$S^1$ worth of choices of real rays in the tensor product (over
$\mathbb C$) of the tangent spaces to $\Sigma$ at these two
punctures). The antibracket $\{-,-\}$ uses twist-attaching in a
similar way.

\subsection{Construction via the Deligne-Mumford compactification}
\label{identification}

\subsubsection{The ``topological'' spectral sequence}

Let $\om_{g}(n)$ denote the Deligne-Mumford (DM) compactification of
$\M_{g}(n)$ and $\om_{g}(n)^{(p)}$ be the moduli space of Riemann
surfaces with at least $p$ \emph{double points} (also known as
\emph{nodes}). Then we have a decreasing filtration
\begin{equation*}
\cdots \subset \om_{g}(n)^{(p+1)} \subset \om_{g}(n)^{(p)} \cdots
\subset \om_{g}(n)^{(0)} = \om_{g}(n).
\end{equation*}
It generates a spectral sequence satisfying
\begin{equation}
  \label{hss}
  E^1_{p,q} = H_{p+q} (\om_{g}(n)^{(-p)}, \om_{g}(n)^{(-p+1)}) \Rightarrow
  H_{\bullet}(\om_{g}(n))
\end{equation}
with $E^1_{p,q} = 0$ unless $q \ge -p \ge 0$ and a differential
{\scriptsize \begin{equation*}
d^1_{p,q} : H_{p+q}(\om_{g}(n)^{(-p)}, \om_{g}(n)^{(-p+1)}) \rightarrow
H_{p+q-1}(\om_{g}(n)^{(-p+1)}, \om_{g}(n)^{(-p+2)}),
  \end{equation*}
}given by the boundary map in the long exact sequence of the triple
\linebreak[4] $(\om_{g}(n)^{(-p)}, \linebreak[0] \om_{g}(n)^{(-p+1)},
\om_{g}(n)^{(-p+2)})$.  We will also utilize a \emph{coboundary map},
the linear dual {\scriptsize
     \begin{equation}
       \label{cobdry}
       (d^1_{p,q})^* : H^{p+q-1}_c
       (\om_{g}(n)^{(-p+1)} \setminus \om_{g}(n)^{(-p+2)}) \rightarrow
  H^{p+q}_c (\om_{g}(n)^{(-p)} \setminus \om_{g}(n)^{(-p+1)})
\end{equation}
}of $d^1_{p,q}$ later on.

After applying Poincar\'{e}-Lefschetz duality and a linear change of
variables $p, q$ to the spectral sequence \eqref{hss}, we get a
cohomological spectral sequence with
\begin{equation}
  \label{cohss}
E_1^{p,q} = H^{q-p} (\om_{g}(n)^{(p)} \setminus \om_{g}(n)^{(p+1)})
\Rightarrow H^{\bullet}(\om_{g}(n))
\end{equation}
with $q \ge p \ge 0$ and a differential
{\small \begin{equation*}
d_1^{p,q} : H^{q-p}(\om_{g}(n)^{(p)} \setminus \om_{g}(n)^{(p+1)})
\rightarrow H^{q-p-1}(\om_{g}(n)^{(p+1)} \setminus
\om_{g}(n)^{(p+2)}).
  \end{equation*}
}This is the \emph{Poincar{\'e} residue map} (see \cite{DII, Getz,
  KSV2}), which might be thought of as a Gysin homomorphism, integration along
the boundary of the tubular neighborhood of $\om_{g}(n)^{(p+1)}$
inside $\om_{g}(n)^{(p)}$. The linear dual is
{\footnotesize \begin{equation}
    \label{Gysin-h}
    (d_1^{p,q})^* :
    H_{q-p-1}(\om_{g}(n)^{(p+1)} \setminus \om_{g}(n)^{(p+2)})
    \rightarrow H_{q-p}(\om_{g}(n)^{(p)} \setminus \om_{g}(n)^{(p+1)})
  \end{equation}
}may also be thought of as a Gysin homomorphism in homology.  We will
now demonstrate that the BV operator \eqref{delta} and antibracket
\eqref{bracket} are essentially these Gysin homomorphisms
$(d_1^{p,q})^*$, see \eqref{Gysin-h}, for $p = 0$. This will imply
that the corresponding operators \eqref{delta-c} and \eqref{bracket-c}
on the compactly supported cohomology are essentially the coboundary
homomorphisms $(d^1_{p,q})^*$, see \eqref{cobdry}, for $p = 0$.

\subsubsection{Identification of the BV operator}
Let
\begin{multline*}
  \alpha_p : \Hb(\om_{g}(n+2)^{(p)} \setminus \om_{g}(n+2)^{(p+1)})\\
  \rightarrow \Hb(\om_{g+1}(n)^{(p+1)} \setminus \om_{g+1}(n)^{(p+2)})
\end{multline*}
be the composition $\alpha_p := \pi_*\rho^!$, where $\pi_*$ is the
pushforward of the map $\pi$ attaching the last two punctures and
$\rho^!$ is the transfer map for $\rho$ from the following diagram of
\'etale morphisms:
\[
\begin{tikzcd}
  (\om_{g,n+2}^{(p)} \setminus \om_{g,n+2}^{(p+1)})/S_n \times
  S_2 \arrow[r, "\pi"] \arrow[d, "\rho"] & \om_{g+1}(n)^{(p+1)}
  \setminus \om_{g+1}(n)^{(p+2)} \\ \om_{g}(n+2)^{(p)} \setminus
  \om_{g}(n+2)^{(p+1)},
\end{tikzcd}
\]
in which $\om_{g,n+2}^{(p)}$ is the moduli space of Riemann surfaces
with $n+2$ labeled punctures and at least $p$ nodes, $S_{n}$ permutes
the first $n$ punctures in $\M_{g,n+2}$, whereas $S_2$ permutes the
last two punctures, and $\rho$ forgets the division of the punctures
into two groups: the first $n$ ones and the last two.


Let us consider the particular case $p=0$, where $\om_g(n+2)^{(0)}
\setminus \om_g(n+2)^{(1)} = \M_g(n+2)$ and $\om_{g+1}(n)^{(1)}
\setminus \om_{g+1}(n)^{(2)}$ is the moduli space of Riemann surfaces
with exactly one node. Then we have
\begin{equation*}
\alpha_0: \Hb(\M_g(n+2)) \rightarrow \Hb(\om_{g+1}(n)^{(1)} \setminus
\om_{g+1}(n)^{(2)}),
\end{equation*}
where $\alpha_0 = \pi_*\rho^!$ for the diagram
\begin{equation}
\label{alpha}
\begin{tikzcd}
  \M_{g,n+2}/S_n \times S_2 \arrow[r, "\pi"] \arrow[d, "\rho"]
    & \om_{g+1}(n)^{(1)} \setminus \om_{g+1}(n)^{(2)} \\
 \M_{g}(n+2)
\end{tikzcd}
\end{equation}
of \'etale morphisms.

\begin{theorem}
  \label{delta-gysin}
\begin{enumerate}
\item The BV operator $\Delta$ as in Eq.~\eqref{delta} is equal to the
  composition $(d_1^{0,\bullet})^* \alpha_0$ below:
\begin{multline*}
  \Hb(\M_{g}(n+2)) \xrightarrow{\alpha_0} \Hb(\om_{g+1}(n)^{(1)}
  \setminus \om_{g+1}(n)^{(2)})\\ \xrightarrow{(d_1^{0,\bullet})^*}
  H_{\bullet +1}(\M_{g+1}(n)).
\end{multline*}
\item The BV operator $\Delta$ as in Eq.~\eqref{delta-c} is equal to the
  composition $(d^1_{0,\bullet})^* \alpha_0$ below:
\begin{multline*}
  \Cb_c(\M_{g}(n+2)) \xrightarrow{\alpha_0} \Cb_c(\om_{g+1}(n)^{(1)}
  \setminus \om_{g+1}(n)^{(2)})\\ \xrightarrow{(d^1_{0,\bullet})^*}
  H^{\bullet +1}_c(\M_{g+1}(n)),
\end{multline*}
where $\alpha_0 = \pi_* \rho^!$ is the morphism induced by the \'etale
morphisms $\pi$ and $\rho$ from \eqref{alpha} on cohomology with
compact support.
  \end{enumerate}
\end{theorem}

\begin{proof}
Claims (1) and (2) are equivalent by Poincar\'e-Lefschetz duality. To
prove claim (1), let us trace what happens under the maps $\rho^!$ and
$\pi_*$ mapking up the map $\alpha_0$, as well as under the map
$(d_1^{0,\bullet})^*$. Given a singular chain of smooth Riemann
surfaces of genus $g$ with $n+2$ unlabeled punctures, the transfer map
$\rho^!$ sums up all possible ways of picking a pair of punctures,
whereas the map $\pi^*$ attaches these two punctures to form a node
and thereby places this chain within the part $\om_{g+1}(n)^{(1)}
\setminus \om_{g+1}(n)^{(2)}$ of the Deligne-Mumford compactification
$\om_{g+1}(n)$ which corresponds to stable curves with exactly one
node. This is exactly what the BV operator $\Delta$ does, except that
twist-gluing at the chosen pair of punctures is replaced so far with
attaching. Now, the map $(d_1^{0,\bullet})^*$ in homology linear dual
to the Poincar\' e residue map $d_1^{0,\bullet}$ in cohomology is the
umkehr map, which associates to a cycle in $\om_{g+1}(n)^{(1)}
\setminus \om_{g+1}(n)^{(2)}$ its pre-image in the boundary of the
tubular neighborhood of $\om_{g+1}(n)^{(1)} \setminus
\om_{g+1}(n)^{(2)}$ inside $\om_{g+1}(n)$. The tubular neighborhood
forms an $S^1$-bundle over $\om_{g+1}(n)^{(1)} \setminus
\om_{g+1}(n)^{(2)}$ and may locally be built out of Riemann surfaces
obtained by twist-gluing at the node of a surface in
$\om_{g+1}(n)^{(1)} \setminus \om_{g+1}(n)^{(2)}$. Thus, composing the
map $(d_1^{0,\bullet})^*$ sends the homology class of stable curves
with one node to the homology class obtained by twist-gluing of stable
curves at the node. This reconciles the composition
$(d_1^{0,\bullet})^* \alpha_0$ with $\Delta$.
\end{proof}

\subsubsection{Identification of the antibracket}

Given that the antibracket has the same nature as the BV operator and
is, namely, the \emph{derived bracket} for the BV operator $\Delta$ on
$S(\fg[-1]) \cong H_\bullet (\M)$, see Remark \ref{physics}, it is not
surprising that there is a similar identification of the antibracket
via the Gysin homomorphisms, coming from the topological spectral sequences.

A diagram analogous to diagram \eqref{alpha}, which defines the
homomorphism $\alpha_0$, is the following pair of \'etale morphisms:
{\scriptsize
  \begin{equation}
    \label{beta}
\begin{tikzcd}
  \M_{g_1,n_1+1}/S_{n_1} \times \M_{g_2,n_2+1}/ S_{n_2}
  \arrow[r, "\pi"] \arrow[d, "\rho"]
& \om_{g_1+g_2}(n_1+n_2)^{(1)}
  \setminus \om_{g_1+g_2}(n_1+n_2)^{(2)}\\
\M_{g_1}(n_1+1) \times \M_{g_2}(n_2+1).
\end{tikzcd}
\end{equation}
}Here $S_{n_i}$ permutes the first $n_i$ punctures in $\M_{g_i,n_i+1}$,
$i = 1, 2$, and $\pi$ attaches the last, $n_1+1$st puncture on the
Riemann surface representing a point in $\M_{g_1,n_1+1}/S_{n_1}$ to
the last, $n_2+1$st puncture on the Riemann surface representing a
point in $ \M_{g_2,n_2+1}/ S_{n_2}$. A diagram analogous to the one
defining $\alpha_p$ may be written down similarly, but we are skipping
it for the sake of simplicity. Define a map
\begin{multline*}
\beta_0: \Hb(\M_{g_1}(n_1+1)) \otimes
\Hb(\M_{g_2}(n_2+1))\\ \rightarrow \Hb( \om_{g_1+g_2}(n_1+n_2)^{(1)}
\setminus \om_{g_1+g_2}(n_1+n_2)^{(2)})
\end{multline*}
as $\beta_0 = \pi_*\rho^!$ for $\pi$ and $\rho$ from the previous
diagram. The following theorem is proven exactly in the same way as Theorem
\ref{delta-gysin}.

\begin{theorem}
\begin{enumerate}
\item
The antibracket $\{-, -\}$ as in Eq.~\eqref{bracket} is equal to the
composition $(d_1^{0,\bullet})^* \beta_0$ below:
\begin{multline*}
\Hb(\M_{g_1}(n_1+1)) \otimes
\Hb(\M_{g_2}(n_2+1))\\ \xrightarrow{\beta_0} \Hb(
\om_{g_1+g_2}(n_1+n_2)^{(1)} \setminus
\om_{g_1+g_2}(n_1+n_2)^{(2)})\\ \xrightarrow{(d_1^{0,\bullet})^*}
H_{\bullet +1}(\M_{g_1+g_2}(n_1+n_2)).
\end{multline*}

\item The antibracket $\{-, -\}$ as in Eq.~\eqref{bracket-c} is equal
  to the composition $(d^1_{0,\bullet})^* \beta_0$ below:
\begin{multline*}
\Cb_c(\M_{g_1}(n_1+1)) \otimes
\Cb_c(\M_{g_2}(n_2+1))\\ \xrightarrow{\beta_0} \Cb_c (
\om_{g_1+g_2}(n_1+n_2)^{(1)} \setminus
\om_{g_1+g_2}(n_1+n_2)^{(2)})\\ \xrightarrow{(d^1_{0,\bullet})^*}
H^{\bullet +1}_c (\M_{g_1+g_2}(n_1+n_2)),
\end{multline*}
where $\beta_0 = \pi_* \rho^!$ is the morphism induced by the \'etale
morphisms $\pi$ and $\rho$ from \eqref{beta} on cohomology with
compact support.
  \end{enumerate}
\end{theorem}

\section{Vanishing results for the BV operator and antibracket}
\label{vanishing}

\subsection{The genus $g=0$ case}
Let us consider the case $g=0$ first, as the BV operator and
antibracket vanish on the moduli space of genus-zero Riemann surfaces
for the trivial reason of it being rationally acyclic.

\begin{theorem}
  \label{genus-zero}
  \[
  H^k (\M_0(n)) =
  \begin{cases}
\Q,\qquad \text{if $k=0$},\\
0,\qquad \text{otherwise}.
\end{cases}
\]
\end{theorem}
This theorem follows from Arnold's computation \cite{Arn, Bri} of the
rational cohomology of the braid group, as observed by Looijenga
\cite{Loo1}. See also a different argument of Westerland
\cite{West}. We also need the following classical result.

\begin{theorem}[Mumford \cite{Mum}]
  $  H^1 (\M_{1,n}) = 0$.
\end{theorem}

These results imply the desired vanishing.

\begin{cor}
$1$.  The BV operator
\[
H_k(\M_{0}(n+2)) \xrightarrow{\Delta} H_{k+1}(\M_{1}(n))
\]
vanishes for all $k \ge 0$, $n \ge 1$. In cohomology with compact
support,
\[
H^l_c(\M_{0}(n+2)) \xrightarrow{\Delta} H^{l+1}_c(\M_{1}(n))
\]
vanishes for all $l \ge 0$, $n \ge 1$.
$2$. The antibracket
{\small
\[
H_{k_1}(\M_{0}(n_1+1)) \otimes H_{k_2} (\M_{0}(n_2+1))
\linebreak[0] \xrightarrow {\{-, -\}} H_{k_1 + k_2 +1}(\M_{0}(n_1 + n_2))
\]
}vanishes for $k_1, k_2 \ge 0$, $n_1, n_2 \ge 1$. In cohomology with
compact support, the antibracket {\small
\[
H_c^{l_1}(\M_{0}(n_1+1)) \otimes H_c^{l_2} (\M_{0}(n_2+1))
\linebreak[0] \xrightarrow {\{-, -\}} H_c^{l_1 + l_2 +1}(\M_{0}(n_1 + n_2))
\]
}vanishes when $l_1, l_2 \ge 0$, $n_1, n_2 \ge 1$.
\end{cor}

\subsection{The genus $g>0$ case}

From \cite{Harer} we have the following bounds on homology:
$$ H_k(\M_{g,n})=0 \quad \text{ for } \quad
\begin{cases}
g=0,\; k>n-3,\\
n=0,\; k>4g-5,\\
n>0,\; g>0,\;k>4g-4+n.
\end{cases}
$$ 
Equivalently,
$$ H_c^l(\M_{g,n})=0 \quad \text{ for } \quad
\begin{cases}
g=0,\; l<n-3,\\
n=0,\; l<2g-1,\\
n>0,\; g>0,\; l<2g-2+n.
\end{cases}
$$

\noindent These above bounds imply the following. For the BV operator
\begin{center}
$H_k(\M_{g}(n+2)) \xrightarrow{\Delta}
H_{k+1}(\M_{g+1}(n))$ 
\end{center}
as in Equation \eqref{delta}, the left-hand side vanishes for
$k>4g-2+n$, $n \ge 0$, while the right-hand side is zero for
$k>4g-1+n$, if $n >0$, and $k > 4g-2$, if $n = 0$. On the other hand,
for the antibracket {\small
\[
H_{k_1}(\M_{g_1}(n_1+1)) \otimes H_{k_2} (\M_{g_2}(n_2+1))
\linebreak[0] \xrightarrow {\{-, -\}} H_{k_1 + k_2 +1}(\M_{g_1 +
  g_2}(n_1 + n_2))
\]
}as in Equation \eqref{bracket}, the left-hand side is zero for $k_1 >
4g_1-3+n_1$ or $k_2 > 4g_2-3+n_2$, but the right-hand side vanishes
for $k_1 + k_2 > 4g_1+4g_2-5+n_1+n_2$, if $n_1+n_2 > 0$, and $k_1 +
k_2 > 4g_1+4g_2-6$, if $n_1 + n_2 = 0$. Therefore, we have

\begin{theorem}
$1$. The BV operator
\[
H_k(\M_{g}(n+2)) \xrightarrow{\Delta} H_{k+1}(\M_{g+1}(n))
\]
vanishes for $k>4g-2+n$, $g > 0$, $n \ge 0$. In cohomology with
compact support,
\[
H^l_c(\M_{g}(n+2)) \xrightarrow{\Delta} H^{l+1}_c(\M_{g+1}(n))
\]
vanishes for $l<2g+n$, $g > 0$, $n \ge 0$. In particular, $\Delta = 0$
on compactly supported cohomology as $g \rightarrow \infty$ or $n \to
\infty$ and $l$ being fixed, \emph{i.e}., stably

$2$. The antibracket
{\small
\[
H_{k_1}(\M_{g_1}(n_1+1)) \otimes H_{k_2} (\M_{g_2}(n_2+1))
\linebreak[0] \xrightarrow {\{-, -\}} H_{k_1 + k_2 +1}(\M_{g_1 +
  g_2}(n_1 + n_2))
\]
}vanishes for $k_1 > 4g_1-3+n_1$ or $k_2 > 4g_2-3+n_2$, $g_1, g_2 >0$
and $n_1, n_2 \ge 0$. In cohomology with compact support, the
antibracket {\small
\[
H_c^{l_1}(\M_{g_1}(n_1+1)) \otimes H_c^{l_2} (\M_{g_2}(n_2+1))
\linebreak[0] \xrightarrow {\{-, -\}} H_c^{l_1 + l_2 +1}(\M_{g_1 +
  g_2}(n_1 + n_2))
\]
}vanishes when $l_1 < 2g_1 -1+n_1$ or $l_2 < 2g_2 -1+n_2$, $g_1, g_2
>0$ and $n_1, n_2 \ge 0$.
\end{theorem}

Likewise, homological stability implies the vanishing of the BV
operator and antibracket within the \emph{stable range} $k <
\frac{2}{3}(g-1)$. Indeed, Harer's stability theorem
\cite{HarerStab}, as improved by Ivanov
\cite{IvaCx, IvaStab, IvaOnt},
Harer himself \cite{HarerPrep}, Boldsen \cite{Bold},
and Randal-Williams \cite{RW}:
\[
H^k (\M_g) \cong H^k (\M_{g+1}) \qquad \text{for $k < \frac{2}{3}(g-1)$},
\]
combined with Madsen and Weiss's proof \cite{MW} of
Mumford's conjecture, stating that
\[
\Q[\kappa_1, \kappa_2, \dots] \to H^\bullet (\M_g),
\]
where $\kappa_i \in H^{2i}(\M_g)$ is the $i$th ``tautological''
$\kappa$ class, $i = 1, 2, \dots$, is an isomorphism in degree $\le
\frac{2}{3}(g-1)$, and Looijenga's relation \cite{Loo} with the case of pointed
Riemann surfaces, which asserts that
\[
H^\bullet (\M_g)[\psi_1, \psi_2, \dots, \psi_n] \to H^\bullet (\M_{g,n}),
\]
where $\psi_i \in H^2(M_g)$ is the $i$th ``tautological'' $\psi$
class, $i = 1, \dots, n$, is an isomorphism in degree $\le
\frac{2}{3}(g-1)$, implies that
\[
\Q[\psi_1, \dots, \psi_n, \kappa_1, \kappa_2, \dots] \to H^\bullet (\M_{g,n})
\]
is an isomorphism in degree $\le \frac{2}{3}(g-1)$. This, in
particular, forces the rational cohomology $H^\bullet (\M_{g,n})$ in
the stable range to be concentrated in even degrees. Taking invariants
of the $S_n$-action on cohomology does not affect these
statements. Now, given that the BV operator and antibracket have
degree 1, we obtain the following vanishing result.

\begin{theorem}
$1$. The BV operator
\[
H_k(\M_{g}(n+2)) \xrightarrow{\Delta} H_{k+1}(\M_{g+1}(n))
\]
vanishes for $k \le \frac{2}{3}(g-1)$, $g > 0$, $n \ge 0$. In
particular, $\Delta = 0$ on homology as $g \rightarrow \infty$ or $n
\to \infty$ and $k$ being fixed, \emph{i.e}., stably. In cohomology
with compact support,
\[
H^l_c(\M_{g}(n+2)) \xrightarrow{\Delta} H^{l+1}_c(\M_{g+1}(n))
\]
vanishes for $l \ge \frac{16}{3}g-\frac{4}{3}+2n$, $g > 0$, $n \ge 0$.

$2$. The antibracket
{\small
\[
H_{k_1}(\M_{g_1}(n_1+1)) \otimes H_{k_2} (\M_{g_2}(n_2+1))
\linebreak[0] \xrightarrow {\{-, -\}} H_{k_1 + k_2 +1}(\M_{g_1 +
  g_2}(n_1 + n_2))
\]
}vanishes for $k_1 \le \frac{2}{3}(g_1-1)$, $k_2 \le
\frac{2}{3}(g_2-1)$, $k_1 + k_2 \le \frac{2}{3}(g_1+g_2)-\frac{5}{3}$,
$g_1, g_2 >0$, and $n_1, n_2 \ge 0$. In cohomology with compact
support, the antibracket {\small
\[
H_c^{l_1}(\M_{g_1}(n_1+1)) \otimes H_c^{l_2} (\M_{g_2}(n_2+1))
\linebreak[0] \xrightarrow {\{-, -\}} H_c^{l_1 + l_2 +1}(\M_{g_1 +
  g_2}(n_1 + n_2))
\]
}vanishes when $l_1 \ge \frac{16}{3}g_1 -\frac{10}{3} +2n_1$, $l_2 \ge
\frac{16}{3}g_2 -\frac{10}{3} +2n_2$, $l_1 + l_2 \ge \frac{16}{3}(g_1
+ g_2) -\frac{19}{3} +2(n_1+n_2)$, $g_1, g_2 >0$, and $n_1, n_2 \ge
0$.
\end{theorem}

\subsection{A nonvanishing example}
Given a vanishing range of the BV operator and antibracket on rational
homology, one may wonder if they vanish identically. Concrete
computations with moduli spaces are quite hard in general, and we do
not have an example of nonvanishing of the antibracket. Here we
present an example of nonvanishing of the BV operator. First off, let
us analyze the differential in the spectral sequence \eqref{cohss}.

\begin{prop}
  \label{nonvan}
  The differential $d_1^{0,6}: H^6(\M_3) \to H^5(\om_3^{(1)} \setminus
  \om_3^{(2)})$ is nontrivial. Moreover, so is its projection to
  $H^5(\M_2(2))$.
\end{prop}

\begin{proof}
It suffices to prove the second statement. By Looijenga's computation
\cite{Loo1} of the rational cohomology of $\M_3$, its
\emph{Poincar\'e-Serre polynomial}, in which the coefficient by $t^k
u^l$ is the dimension of the subquotient $H^k(\M_3)$ of weight $l$, is
equal to $1 + t^2u^2 + t^6u^{12}$. This implies that $H^6(\M_3) \cong
\Q$ of weight 12. Since the weight is not equal to the cohomology
degree, the corresponding element in $E_1^{0,6}$ will not survive to
$H^6(\om_3)$ in the limit $E_{\infty}^{0,6}$. If all the differentials
$d_r^{0,6}$, $r \ge 1$, were zero on $H^6 (\M_3)$, then $H^6 (\M_3)$
would contribute nontrivially to $E_\infty^{0,6} \subset H^6 (\om_3)$
and thereby have weight 6, which would be contradiction. The plan is
to show that all the higher differentials $d_r^{0,6}$, $r \ge 2$,
vanish. This would force $d_1^{0,6}$ to be nontrivial.

Now, looking at different components of the boundary $\bd \M_3$, we
find that they all have a trivial cohomology group $H^5$, except
possibly $\M_2(2)$. Indeed, $\om_3^{(1)} \setminus \om_3^{(2)} =
\M_2(2) \coprod \M_{2,1} \times \M_{1,1}$. The space $\M_{1,1}$ of
elliptic curves is known to be isomorphic to the affine line $\C$,
whereas the following argument, borrowed from Dan Petersen
\cite{Peters}, shows that $\M_{2,1}$ has the rational homology of
$\CPP$. Indeed, consider the Leray-Serre spectral sequence for the
forgetful map $\pi: \M_{2,1} \to \M_2$. The base $\M_2$ is isomorphic
to $\M_0(6)$, because every curve of genus 2 is hyperelliptic. We know
from Theorem \ref{genus-zero} that it has the cohomology of a
singleton. On the other hand, the local systems $R^0 \pi_* \Q$ and
$R^2 \pi_* \Q$ are trivial, while $R^1 \pi_* \Q$ has no rational
cohomology, as the curve representing a point of $\M_2$ has the
hyperelliptic involution, which acts on the fiber of $R^1 \pi_* \Q$ by
$-1$.

We conclude that the projection of $d_1^{0,6}: H^6(\M_3) \to
H^5(\om_3^{(1)} \setminus \om_3^{(2)})$ onto $H^5(\M_{2,1} \times
\M_{1,1}) = 0$ must be zero. If we show that the higher differentials
$d_r^{0,6}$, $r \ge 2$, vanish, it will imply that the projection of
$d_1^{0,6}: H^6(\M_3) \to H^5(\om_3^{(1)} \setminus \om_3^{(2)})$ to
$H^5(\M_2(2))$ is nontrivial.

For each $r \ge 2$, the higher differential $d_r^{0,6}$ maps
$E_r^{0,6}$, which is a subspace of $E_1^{0,6}$, to $E_r^{r,5+r}$,
which is a subquotient of $E_1^{r,5+r}$. We claim that all these terms
$E_1^{r,5+r}$ are zero. For $r = 2$, the term $E_1^{r,5+r} = E_1^{2,7}
= H^5 (\om_3^{(2)} \setminus \om_3^{(3)})$ is the direct sum of the
cohomology groups of connected components of $\om_3^{(2)} \setminus
\om_3^{(3)})$. These components are quotients under finite group
actions of the following spaces:
\[
\M_{1,4}, \quad \M_{2,1} \times \M_{0,3}, \quad \M_{1,3} \times
\M_{1,1}, \quad \M_{1,2} \times \M_{1,1} \times \M_{1,1}.
\]
Looking at the forgetful map $\M_{1,4} \to \M_{1,1}$, which is
topologically a fiber bundle with fiber of the homotopy type of a
three-dimensional CW complex and base having the homotopy type of a
point, we see that $H^5(\M_{1,4}) = 0$. We have already seen that
$\M_{2,1}$ has the rational homology of $\CPP$, and so does $\M_{2,1}
\times \M_{0,3}$. Similar to $\M_{1,4}$, the space $\M_{1,3}$ has the
homotopy type of a CW complex of dimension two, and so does $\M_{1,3}
\times \M_{1,1}$. A similiar argument works for $\M_{1,2} \times
\M_{1,1} \times \M_{1,1}$.

Analyzing similarly the groups $E_1^{r,5+r} = H^5 (\om_3^{(r)}
\setminus \om_3^{(r+1)})$ for $r \ge 3$, we quickly see that all of
them vanish for dimensional reasons.
\end{proof}

\begin{cor}
  \label{counter}
  The BV operator
  \[
  \Delta: H_5(\M_2(2)) \to H_6(\M_3)
  \]
  does not vanish and is moreover an isomorphism between these
  one-di\-men\-sion\-al vector spaces over $\Q$.
\end{cor}

\begin{proof}
Working at the dual, cohomology level in the proof of Proposition
\ref{nonvan}, we have seen that $H^6(\M_3)$ is
one-di\-men\-sion\-al. We also have $H^5(\M_2(2)) = \Q$, given
Tommasi's computation \cite[Corollary III.2.2]{tommasi:thesis} of the
Poincar\'{e}-Serre polynomial of $\M_2(2)$ as $1+t^2u^2 + t^5
u^{10}$.

Thus, to prove our claim, we only need to see that $\Delta \ne 0$. By
Theorem \ref{delta-gysin}, we know that $\Delta = (d_1^{0,6})^*
\alpha_0$. Recall that $\alpha_0 = \pi_* \rho^!$, where $\pi$ and
$\rho$ were given in diagram \eqref{alpha}. Note that in this
particular case, $\rho = \operatorname{id}$ and $\pi$ is the inclusion
of $\M_2(2)$ as a connected component of $\om_3^{(1)} \setminus
\om_3^{(2)}$. Thus, $\Delta = (d_1^{0,6})^* \pi_*$, which is exactly
the linear dual of the projection of $d_1^{0,6}$ to $H^5(\M_2(2))$,
whose nonvanishing is proven in Proposition \ref{nonvan}.
  \end{proof}



\begin{thebibliography}{9999}


\bibitem{Arn} V. I. Arnol’d, \emph{Certain topological invariants of
  algebraic functions}, (Russian) Trudy Moskov.\ Mat.\ Obsc. 21 (1970)
  27--46; translation in Trans.\ Moscow Math.\ Soc. 21 (1970) 30--52.

\bibitem{barannikov} S.~Barannikov, \emph{Modular operads and
  {B}atalin-{V}ilkovisky geometry}, Int.  Math. Res. Not. IMRN (2007),
  no.~19, Art.\ ID rnm075, 31.

\bibitem{Bold} S. Boldsen, \emph{Improved homological stability for
  the mapping class group with integral or twisted coefficients},
  Mathematische Zeitschrift 270 (2012), no. 1--2, 297--329.

\bibitem{Bri} E. Brieskorn, \emph{Sur les groupes de tresses},
  Séminaire Bourbaki 14 (1971--1972) 21--44.

\bibitem{CFL} K. Cieliebak, K. Fukaya, J. Latschev, \emph{Homological
  algebra related to surfaces with boundary}, Preprint,
  \url{arXiv:1508.02741}.

\bibitem{CL} K. Cieliebak, J. Latschev, \emph{The role of string
  topology in symplectic field theory}, New perspectives and
  challenges in symplectic field theory, CRM Proc.\ Lecture Notes, 49,
  Amer.\ Math.\ Soc., Providence, RI, 2009, pp.\ 113--146.

\bibitem{Cos} K.~Costello, {\em The partition function of a
  topological field theory}, Journal of Topology 2 (2009) 779--822.

\bibitem{DII} P.~Deligne, {\em Th{\'e}orie de Hodge : II},
  Publications Math{\'e}matiques de l'I.H.\'{E}.S., 40 (1971) 5--57.

\bibitem{DM} M. Doubek, M. Markl, \emph{Open-closed modular operads,
  the Cardy condition and string field theory}, J. Noncommut. Geom. 12
  (2018), no.\ 4, 1359--1424.

\bibitem{Eliashberg-Givental-Hofer} Y. Eliashberg, A. Givental,
  H. Hofer, \emph{Introduction to symplectic field theory}, GAFA 2000
  (Tel Aviv, 1999), Geom. Funct. Anal. Special Volume, Part II (2000),
  560--673.

\bibitem{Getz} E.~Getzler, {\em Operads and moduli spaces of genus 0
  Riemann surfaces}, The moduli space of curves, Birkh\"{a}user
  Boston, 1995, pp.\ 199--230.


\bibitem{Getz2} E. Getzler, {\em Topological Recursion Relations in
  Genus 2}, Integrable systems and algebraic geometry (Kobe/Kyoto,
  1997), World Sci. Publishing, River Edge, NJ, 1998, pp. 73--106.


\bibitem{HarerStab} J.~L. Harer, \emph{Stability of the homology of
  the mapping class groups of orientable surfaces}, Ann.\ of
  Math. (2), 121(2), 215--249, 1985.

\bibitem{Harer} J.~L. Harer, \emph{The virtual cohomological dimension
  of the mapping class group of an orientable surface},
  Invent.\ Math. 84 (1986) 157--176.

\bibitem{HarerPrep} J.~L. Harer, \emph{Improved stability for the
  homology of the mapping class groups of surfaces}, Preprint, 1993.

\bibitem{HVZ} E. Harrelson, A. Voronov, J. Zuniga, \emph{Open-closed
  moduli spaces and related algebraic structures},
  Lett. Math. Phys. 94 (2010), no.\ 1, 1--26.

\bibitem{IvaCx} N. V. Ivanov, \emph{Complexes of curves and the
  Teichm\"{u}ller modular group}. (Russian) Uspekhi Mat. Nauk, 42(3),
  49--91, 1987; translation in Russian Math.\ Surveys, 42(3), 55--107,
  1987.

\bibitem{IvaStab} N. V. Ivanov, \emph{Stabilization of the homology of
  Teichm\"{u}ller modular groups}. (Russian) Algebra i Analiz, 1(3),
  110--126, 1989; translation in Leningrad Math.\ J., 1(3), 675--691,
  1990.

\bibitem{IvaOnt} N. V. Ivanov, \emph{On the homology stability for
  Teichm\"{u}ller modular groups: closed surfaces and twisted
  coefficients}, Mapping class groups and moduli spaces of Riemann
  surfaces (G\"ottingen, 1991/Seattle, WA, 1991), Contemp.\ Math.,
  150, Amer.\ Math.\ Soc., Providence, RI, 1993, pp.\ 149--194.

\bibitem{JM} B. Jur\v{c}o, K. M\"{u}nster, \emph{Type II superstring
  field theory: geometric approach and operadic description}, J. High
  Energy Phys. 2013, no. 4, 126, front matter + 36 pp.

\bibitem{KKP} L. Katzarkov, M. Kontsevich, T. Pantev, \emph{Hodge
  theoretic aspects of mirror symmetry}, From Hodge theory to
  integrability and TQFT tt*-geometry, Proc.\ Sympos.\ Pure Math., 78,
  Amer.\ Math.\ Soc., Providence, RI, 2008, pp.\ 87--174.
  
\bibitem{Kauf} R. Kaufmann, \emph{Open/closed string topology and
  moduli space actions via open/closed Hochschild actions}, SIGMA
  Symmetry Integrability Geom. Methods Appl. 6 (2010), Paper 036, 33
  pp.

\bibitem{KWZ} R. Kaufmann, B. Ward, J. Zuniga, \emph{The odd origin of
  Gerstenhaber brackets, Batalin-Vilkovisky operators, and master
  equations}, J. Math. Phys. 56 (2015), no.\ 10, 103504, 40 pp.


\bibitem{KSV} T.~Kimura, J.~Stasheff, A.~A. Voronov, {\em On operad
  structures of moduli spaces and string theory},
  Comm.\ Math.\ Phys. 171 (1995) 1--25.

\bibitem{KSV2} T.~Kimura, J.~Stasheff, A.~A. Voronov, {\em Homology of
  moduli spaces of curves and commutative homotopy algebras}, The
  Gelfand Mathematical Seminars, 1993--1995, Birkh\"{a}user Boston,
  1996, pp.\ 151--170.

\bibitem{LLS} C.~Li, S.~Li, K.~Saito, \emph{Primitive forms via
  polyvector fields}, Preprint, \url{arXiv:1311.1659 [math.AG]}.

\bibitem{Loo1} E. Looijenga, \emph{Cohomology of $\M_3$ and $\M_3^1$},
  Mapping class groups and moduli spaces of Riemann surfaces
  (G\"ottingen, 1991/Seattle, WA, 1991), Contemp.\ Math., 150,
  Amer.\ Math.\ Soc., Providence, RI, 1993, pp.\ 205--228.

\bibitem{Loo} E. Looijenga, \emph{Stable cohomology of the mapping
  class group with symplectic coefficients and of the universal
  Abel-Jacobi map}, Journal of Algebraic Geometry 5 (1996), no.\ 1,
  135--150.


\bibitem{MW} I. Madsen and M. Weiss, \emph{The stable moduli space of
  Riemann surfaces: Mumford’s conjecture}, Annals of Mathematics
  (2007), 843--941.


\bibitem{Markl} M. Markl, \emph{Loop Homotopy Algebras in Closed
  String Field Theory}, Commun. Math. Phys. 221, 367--384 (2001).


\bibitem{Mum} D. Mumford, \emph{Abelian quotients of the
  Teichm\"{u}ller modular group}, Journal d'Analyse Math\'{e}matique
  18 (1967), no. 1, 227--244.


\bibitem{MuSA} K. M\"{u}nster, I. Sachs, \emph{Quantum open-closed
  homotopy algebra and string field theory}, Comm. Math. Phys. 321
  (2013), no. 3, 769--801.


  \bibitem{Peters} D.~Petersen, \emph{Betti numbers of moduli spaces
        of smooth Riemann surfaces}, version: 2012-02-15,
      \url{https://mathoverflow.net/q/88515}.

\bibitem{RW} O.~Randal-Williams, \emph{Resolutions of moduli spaces
  and homological stability}, JEMS 18 (2016), no.\ 1, 1--81.

\bibitem{saito:moduli} K.~Saito, \emph{Moduli space for Fuchsian
  groups}, Algebraic analysis, Vol.~II, Academic Press, Boston, 1988,
  pp.\ 735--787.

\bibitem{saito:higher} K.~Saito, \emph{Higher Eichler integrals and
  vector bundles over the moduli of spinned Riemann surfaces},
  Prospects in complex geometry (Katata and Kyoto, 1989), Lecture
  Notes in Math., 1468, Springer, Berlin, 1991, pp.\ 408--421.

\bibitem{saito:the} K.~Saito, \emph{The Teichmüller space and a
  certain modular function from a viewpoint of group representations},
  Algebraic geometry and related topics (Inchon, 1992),
  Conf.\ Proc.\ Lecture Notes Algebraic Geom., I, Int.\ Press,
  Cambridge, MA, 1993, pp.\ 41--88.

\bibitem{saito:algebraic} K.~Saito, \emph{Algebraic representations of
  Teichmüller space}, Workshop on Geometry and Topology (Hanoi, 1993),
  Kodai Math.\ J. 17 (1994), no.\ 3, 609--626.
  
\bibitem{Sen-Zw} A. Sen, B. Zwiebach, \emph{Quantum background
   independence of closed-string field theory}, Nuclear Phys. B 423
   (1994), no. 2--3, 580--630.

 \bibitem{terilla:smoothness} J. Terilla, \emph{Smoothness theorem for
   differential BV algebras}, J. Topol. 1 (2008), no.\ 3, 693--702.

 \bibitem{tommasi:thesis} O. Tommasi, \emph{Geometry of Discriminants
   and Cohomology of Moduli Spaces}, Ph.D. thesis, Radboud
   Universiteit Nijmegen, The Netherlands, 2005, 89 pp.,
   \url{http://hdl.handle.net/2066/27027}.

   \bibitem{QDT2} A.~A. Voronov, \emph{Quantizing deformation theory
  {II}}, {P}reprint IPMU18-0117, Kavli IPMU, 2018,
  \url{arXiv:1806.11197 [math.QA]}, to appear in Pure Appl.\ Math.\ Q.

\bibitem{West} C. Westerland, \emph{Configuration spaces in topology
  and geometry}, Gazette of the Australian Mathematical Society, 38
  (2011), no. 5, 279--283.

\bibitem{Witt} E. Witten, \emph{On background-independent open-string
  field theory}, Phys.\ Rev.\ D (3) 46 (1992), no.\ 12, 5467--5473.


\bibitem{Zw} B.~Zwiebach, {\em Closed string field theory: Quantum
  action and the Batalin-Vilkovisky master equation}, Nuclear Physics
  B 390 (1993) 33--152.



\end{thebibliography}
\end{document}